\documentclass[VANCOUVER,STIX1COL]{WileyNJD-v2}
\usepackage{moreverb}

\newcommand\BibTeX{{\rmfamily B\kern-.05em \textsc{i\kern-.025em b}\kern-.08em
		T\kern-.1667em\lower.7ex\hbox{E}\kern-.125emX}}

\articletype{Article Type}%

\received{<day> <Month>, <year>}
\revised{<day> <Month>, <year>}
\accepted{<day> <Month>, <year>}

\usepackage{bm}
\usepackage{amsthm}
\usepackage{mathtools}
\usepackage{amsmath}

\usepackage{amsfonts}

\usepackage{mathtools}

\usepackage{graphicx}
\usepackage{epstopdf}
\usepackage{afterpage}
\usepackage{url}
\usepackage{caption}

\usepackage{lscape}
\usepackage{longtable}

\newcommand{\bfS}{\mathbf{S}}

\newcommand{\bfI}{\mathbf{I}}

\newcommand{\bfH}{\mathbf{H}}
\newcommand{\bfHh}{\widehat{\mathbf{{H}}}}
\newcommand{\bfR}{\mathbf{R}}

\newcommand{\bfRh}{\widehat{\mathbf{{R}}}}
\newcommand{\Sh}{\widehat{\bfS}}
\newcommand{\bfC}{\mathbf{C}}
\newcommand{\bfD}{\mathbf{D}}
\newcommand{\bfB}{\mathbf{B}}
\newcommand{\bfF}{\mathbf{F}}
\newcommand{\bfM}{\mathbf{M}}
\newcommand{\bfx}{\mathbf{x}}
\newcommand{\bfz}{\mathbf{z}}
\newcommand{\bfy}{\mathbf{y}}
\newcommand{\bfA}{\mathbf{A}}

\newcommand{\bfd}{\mathbf{d}}

\newtheorem{mydef}{Definition}
\newtheorem{cor}{Corollary}
\newtheorem*{mainass}{Key Assumption}

\begin{document}

	\date{}

	\title{{\color{black}New bounds on the condition number of the Hessian of the} preconditioned variational data assimilation problem}
	\author[1,2,3]{Jemima M. Tabeart*}
	\author[1]{Sarah L. Dance}
	\author[1,2]{Amos S. Lawless}
	\author[1,2]{Nancy K. Nichols}
	\author[1,4]{Joanne A. Waller}

	\authormark{Jemima M. Tabeart \textsc{et al}}

	\address[1]{\orgdiv{School Of Mathematical, Physical and Computational Sciences}, \orgname{University of Reading}, \orgaddress \country{UK}}
	
	\address[2]{ \orgname{National Centre for Earth Observation}, \orgaddress{\state{Reading}, \country{UK}}}
	
	\address[3]{Now at \orgdiv{School of Mathematics}, \orgname{University of Edinburgh}, \orgaddress{\country{UK}}}
	
	\address[4]{Now at \orgname{MetOffice@Reading},  \country{UK}}

	\corres{*Jemima M. Tabeart, School of Mathematics, University of Edinburgh, James Clerk Maxwell Building, The King's Buildings, Peter Guthrie Tait Road, Edinburgh, EH9 3FD, United Kingdom, \email{jemima.tabeart@ed.ac.uk}}

	\graphicspath{{/home/jemima/Dropbox/UoR/PhD/Thesis/Figures/}}
	
	\abstract[Abstract]{Data assimilation algorithms combine prior and observational information, weighted by their respective uncertainties, to obtain the most likely posterior of a dynamical system. In variational data assimilation
		the posterior is computed by solving a nonlinear least squares problem.
		Many numerical weather prediction (NWP) centres use full observation error covariance (OEC) weighting matrices, which can slow convergence of the data assimilation procedure.
		Previous work revealed the importance of the minimum eigenvalue of the
		OEC matrix for conditioning and convergence of
		the unpreconditioned data assimilation problem. 
		In this paper we examine the use of correlated OEC matrices in the preconditioned data assimilation problem for the first time. We consider the case where there are more state variables than observations, which is typical for applications with sparse measurements e.g. NWP and remote sensing.
		We find that similarly to the unpreconditioned problem, the minimum eigenvalue of
		the OEC matrix appears in new bounds on the condition number of the Hessian of the preconditioned objective function. Numerical experiments reveal that the condition number of the Hessian is minimised when the background and observation lengthscales are equal.
		This contrasts with the unpreconditioned case, where decreasing the observation error lengthscale always improves conditioning.  Conjugate gradient experiments {\color{blue}show} that in this framework the condition number of the Hessian is a good proxy for convergence. Eigenvalue clustering explains cases where convergence is faster than expected.} 

	
	\maketitle

	\section{Introduction}

{\color{black}Data assimilation algorithms combine observations of a dynamical system{\color{black}, $\bfy_i\in\mathbb{R}^{p_i}$ at times $t_i$,} with prior information from a model{\color{black}, $\bfx^b \in \mathbb{R}^N$} to find {\color{black}$\bfx_i\in \mathbb{R}^N$}, the most likely state of the system at {\color{black}time $t_i$}.
	In variational data assimilation
	the posterior is computed by solving a nonlinear least squares problem.		In this paper we examine the effect of using correlated observation error covariance matrices on the convergence of the preconditioned variational data assimilation problem. We develop new bounds on the condition number of the Hessian of the linearised  preconditioned objective function. 
	Numerical experiments allow us to compare the bounds to the computed condition number. We also investigate the relationship between conditioning, the full spectrum of the Hessian and convergence of a linear data assimilation test problem to assess the suitability of using the condition number of the  Hessian as a proxy for convergence in this setting.

	We now define the variational data assimilation objective function of interest for this paper.} 
{\color{black} For a time window $[t_0,t_n]$, we	let $\bfx^t_i\in \mathbb{R}^N$ be the true state of the dynamical system of interest at time $t_i$, where $N$ is the number of state variables. 	
	The prior, or background state, is valid at the initial time $t_0$ and can be written as an approximation to the true state as $\bfx^b = \bfx^t_0 + \epsilon^b$. We assume that the background errors $\epsilon^b\sim\mathcal{N} (0,\bfB) $, where $\bfB\in\mathbb{R}^{N\times N}$ is the background error covariance matrix. 
	As observations can be made at different locations, or of different variables to those in the state vector $\bfx_i$,  we define an observation operator $h_i:\mathbb{R}^{N} \rightarrow \mathbb{R}^{p_i}$ which maps from state variable space to observation space at time $t_i$. 
	Observations $\bfy_i\in\mathbb{R}^{p_i}$ at time $t_i$ are similarly expressed as $\bfy_i = h_i[\bfx^t_i]+\epsilon_i$ for $i=0,\dots,n$: the sum of the model equivalent $h_i[\bfx^t_i]$ and an observation error  $\epsilon_i\sim \mathcal{N}(0,\bfR_i)$,  where  $\bfR_i\in\mathbb{R}^{p_i\times p_i}$ are the observation error covariance (OEC) matrices.		We additionally assume that the observation and background errors are mutually uncorrelated. 
	The total number of observations across the whole time window is given by $p=\sum_{i=0}^{n}p_i$. 
	The state $\bfx_{i-1}$ at time $t_{i-1}$  is propagated to the next observation time $t_{i}$ using a nonlinear forecast model operator, $\mathcal{M}$, to obtain
\begin{equation}\label{eq:modelconstraint}
	\bfx_i = \mathcal{M}(t_{i-1},t_i,\bfx_{i-1}).
	\end{equation}
	
	In variational data assimilation the analysis, $\bfx_0$, or most likely {\color{black} state at the initial time $t_0$,}  minimises the  full 4D-Var objective function, given by 
	\begin{equation} \label{eq:6:CostFn}
	J(\bfx_0)=\frac{1}{2}(\bfx_0-\bfx^b)^T\mathbf{B}^{-1}(\bfx_0-\bfx^b)+\frac{1}{2}\sum_{i=0}^{n}(\bfy_i-h_i[\bfx_i])^T\bfR_i^{-1}(\bfy_i-h_i[\bfx_i]). 
	\end{equation}
	{\color{black} In applications such as numerical weather prediction (NWP), the non-linear objective function \eqref{eq:6:CostFn} is typically minimised using an iterative method.} }
The most common implementation is the incremental formulation, which solves the variational data assimilation problem via a small number of nonlinear outer loops, and a larger number of inner loop iterations which minimise a linearised least squares problem \cite{Courtier94}. This procedure is equivalent to a Gauss-Newton method \cite{GrattonS.2007AGMf,lawless05,lawless05b} {\color{black} and will be presented in Section 2}.

For many systems in the geosciences and neurosciences\cite{Carrassi18,SchiffStevenJ2011NCET}  the number of state variables, {\color{black}$N$,} can be of the order of $10^9$. 
{\color{black}In this paper we consider the case where the number of state variables is greater than the number of observations, {\color{black}i.e. $N>p$}, an assumption which holds for applications with sparse measurement data.
	The large dimension of the state space motivates the use of
	a control variable transform (CVT) to model the background error covariance matrix, $\bfB$, implicitly \cite{bannister08}.  The CVT uses the square root of {\color{black}$\bfB$} as a variable transform to obtain a modified objective function \cite[Sec 9.1]{lewis06}, and can be interpreted as a form of preconditioning.  The transformation diagonalises the weighting on the first term of {\color{black}\eqref{eq:6:CostFn}}, making the transformed state variables uncorrelated. 
	We refer to the incremental variational problem with the CVT as the preconditioned data assimilation problem for the remainder of this paper.

	As the inner iterations of the incremental 4D-Var algorithm solve a linear least squares problem, the conjugate gradient method can be used for the minimisation of the linearised objective function \cite{Fisher98,Liu2018,Tremolet07}.  
	Convergence of a conjugate gradient method can be bounded by the condition number of the Hessian of the objective function \cite{Gill86,golub96,haben11c}; {\color{black}therefore} the condition number of the linearised Hessian {\color{black} can be considered} as a proxy to study how changes to a data assimilation method are likely to affect convergence of the inner loop. 
	{\color{black}The Hessian of the linearised preconditioned objective function is given by
		\begin{equation}\label{eq:Hessian}
		\Sh = \bfI+\sum_{i=0}^n \bfB^{-1/2}\bfM_i^T\bfH_i^T\bfR_i^{-1}\bfH_i\bfM_i\bfB^{-1/2},
		\end{equation}
		where $\bfH_i\in \mathbb{R}^{p_i\times N}$ is the linearised observation operator, and $\bfM_i\in\mathbb{R}^{N\times N}$ is the linearisation of the model operator \eqref{eq:modelconstraint}.} The Hessian \eqref{eq:Hessian}
	is a low rank update of the identity matrix, and hence is typically better conditioned than the Hessian corresponding to the unpreconditioned problem (as its minimum eigenvalue is one).
	However, {\color{black}the distribution of the full spectrum, and not just the extreme eigenvalues, is important for the conjugate gradient method (see Theorems 38.3, 38.5 of Trefethen and Bau \cite{TrefethenLloydN.LloydNicholas1997Nla}, and Theorem 38.4 of Gill, Murray and Wright \cite{Gill86})}. In this paper we will therefore consider how the condition number of the Hessian relates to convergence of the conjugate gradient method in an idealised  numerical framework{\color{black}, and examine the distribution of  the full spectrum of \eqref{eq:Hessian}}. 
	
	In recent years there has been a rise in the introduction of correlated OEC matrices {\color{black}($\bfR_i$ in \eqref{eq:6:CostFn})} at NWP centres (e.g. Bormann et al. \cite{bormann16}, Weston et al. \cite{weston14}, Janji\'c et al.\cite{janjic17}). 
	The use of correlated OEC matrices  brings benefit to applications by allowing users to include more observations \cite{stewart13,simonin18} at higher resolutions. Correlated OEC matrices also lead to greater information content of observations, particularly on smaller scales \cite{stewart08,stewart13,fowler17,rainwater15}. However, the move from uncorrelated (diagonal) to correlated (full) covariance matrices has caused problems with the convergence of the data assimilation procedure in experiments at NWP centres \cite{weston11,weston14,tabeart17b}.  Previous studies of the conditioning of the preconditioned Hessian have focussed on the case of uncorrelated OEC matrices \cite{haben11c,haben11b}. In this paper we extend this theory to the case of correlated OEC matrices.
	
	Tabeart et al. \cite{tabeart17a} considered the effect of using correlated (full) OEC matrices within the unpreconditioned data assimilation problem. 
	The minimum eigenvalue of the correlated OEC matrix was found to be important in determining the conditioning of the Hessian of the objective function both theoretically and numerically.
	The condition number of the Hessian was found to be a good proxy for convergence in this framework.
Haben et al. \cite{haben11b,haben11c} developed bounds on the condition number of the Hessian for both the unpreconditioned and preconditioned problems in the case of uncorrelated (diagonal) OEC matrices. 
In the preconditioned case, reducing the observation error variance increases both the bounds on and numerical value of the condition number of the Hessian. The choice of observation network was also shown to be important for determining the conditioning and convergence of the preconditioned problem.

	In this paper we consider the conditioning of the preconditioned variational data assimilation problem in the case of correlated OEC matrices.  We extend the analysis of Tabeart et al.\cite{tabeart17a} to the preconditioned case {\color{black} where there are fewer observations than state variables,} {\color{black} i.e. $p<N$}. We begin in Section~\ref{sec:6:Background} by defining the problem and introducing existing mathematical results relating to conditioning. In Section~\ref{sec:6:PrecondTheory} we present new theoretical bounds on the condition number of the preconditioned Hessian in terms of its constituent matrices.  In Section~\ref{sec:NumFramework} we introduce the numerical framework that will be used for our experiments. We present the results of these experiments and related discussion in Section~\ref{sec:6:Numerics}. These experiments reveal the ratio between background and observation error correlation lengthscales strongly influences the conditioning of the Hessian, with minimum condition numbers occurring when the two lengthscales are equal. This contrasts with the unpreconditioned case, where the condition number of the Hessian could always be reduced by decreasing the lengthscale of the observation error covariances. 
	We find cases where the new bounds represent the qualitative behaviour of the conditioning well, as well as cases where bounds from Haben \cite{haben11c} are tighter.  
	For many cases the condition number of the Hessian is a good proxy for convergence of a conjugate gradient method. Cases where convergence is much faster than expected can be explained by a single large eigenvalue with the remainder clustering around unity.  
	Our conclusions are presented in Section \ref{sec:6:Conclusions}. 
	\section{The preconditioned variational data assimilation problem}\label{sec:6:Background}
	
	\subsection{The control variable transform formulation of the data assimilation problem}\label{sec:6:DA}

	{\color{black}In this section we} define the preconditioned 4D-Var data assimilation problem and introduce {\color{black} further} notation that will be used in this paper.  
	We recall \cite{tabeart17c} that covariance matrices can be decomposed as $\bfB = \Sigma_B\widetilde{\bfB}\Sigma_B$, and $\bfR_i = \Sigma_{R_i}\widetilde{\bfR}_i\Sigma_{R_i}$ where $\Sigma_B,\Sigma_{R_i}$ are diagonal matrices containing standard deviations, and $\widetilde{\bfB},\widetilde{\bfR}_i$ are correlation matrices with unit entries on the diagonal. By definition covariance matrices are symmetric positive semi-definite. However, we will assume in what follows that $\widetilde{\bfB}, \widetilde{\mathbf{R}}_i,\bfB$ and $\bfR_i$ are strictly positive definite, and therefore their inverses are well-defined.

	We now derive the linearised incremental objective function. In this formulation instead of finding the state which minimises the objective function \eqref{eq:6:CostFn} directly, subject to the model constraint \eqref{eq:modelconstraint},
	we minimise a sequence of linearisations of the objective function to obtain a sequence of increments to the background, $\bfx^b$. 
	Typically this is done via a series of outer loops, where the forecast model and observation operators are linearised about the current best estimate of $\bfx_0$. 

	For the $l$th outer loop  we define $\bfx_0^{(l+1)} = \bfx_0^{(l)}+\delta\bfx_0^{(l)}$. We then consider the Taylor expansion of $\mathcal{M}(t_{i-1},t_i;\bfx_{i-1}^{(l)})$ 
	and obtain the linearisation
$\delta\bfx_i^{(l)} =  \bfM_{i}\delta\bfx_{i-1}^{(l)}$
	{where $\bfM_{i}\in\mathbb{R}^{N\times N}$ is the linearised model operator at time $t_i$, linearised about the model forecast initialised at $\bfx_0^{(l)}$.} Finally we denote $\delta\bfx^{(l)}_b = \bfx^b-\bfx_0^{(l)}$, with $\bfx_{0}^{(0)}= \bfx^b$ and $\delta\bfx_0^{(0)}=0$.
	
	Similarly, expanding $h_i[\bfx_i]$ about $\bfx_i^{(l)}$ we obtain the linearisation
	{\color{black}$	h_i[\bfx_i^{(l)}+\delta\bfx_i^{(l)}] \approx h_i[\bfx_i^{(l)}] +\bfH_i\delta\bfx_i^{(l)}$}
	where $\bfH_i\in\mathbb{R}^{p_i\times N}$ is the linearised observation operator at time $t_i$   linearised about $\bfx_i^{(l)}$.

	We then write the linearised objective function in terms of $\delta\bfx_0^{(l)}$,
	\begin{equation}\label{eq:6:IncCostFn}
	\widetilde{J}(\delta\bfx_0^{(l)})=\frac{1}{2}(\delta\bfx_0^{(l)} - \delta\bfx_b^{(l)})^T\mathbf{B}^{-1}(\delta\bfx_0^{(l)} - \delta\bfx_b^{(l)})+\frac{1}{2}\sum_{i=0}^{n}(\bfd_i^{(l)} - \bfH_i\delta\bfx_i^{(l)} )^T\bfR_i^{-1}(\bfd_i^{(l)} - \bfH_i\delta\bfx_i^{(l)})
	\end{equation}
	where $\bfd_i^{(l)} = \bfy_i - h_i[\bfx^{(l)}_i]$ are the innovation vectors. 	These measure the misfit between the observations and the linearised state, using the full nonlinear observation operator.
	
	In order to simplify the notation in what follows we can group the linearised forecast model and observation operator terms together as a single linear operator. 
	We  define the generalised observation operator as 	\begin{equation}\label{eq:6:GenH}
	\widehat{\bfH} = \left[\bfH_0^T,(\bfH_1\widehat{\bfM}_1)^T,\dots,(\bfH_n\widehat{\bfM}_n)^T\right]^T \in \mathbb{R}^{N(n+1)\times p(n+1)},
	\end{equation}
	where the linearised forward model from time $t_0$ to time $t_i$ is given by
	\begin{equation}
	\widehat{\bfM}_i\delta\bfx_0^{(l)} 
	=\bfM_i\dots\bfM_1\delta\bfx_0^{(l)}.
	\end{equation}
	Finally we let $\widehat{\bfR} \in\mathbb{R}^{p\times p}$ denote the block diagonal matrix with the $i$th block consisting of $\bfR_i$.	
	This allows us to write the Hessian of the linearised objective function, \eqref{eq:6:IncCostFn}, in the simplified form
	\begin{equation}\label{eq:Hessianunprecond}
	\bfS = \bfB^{-1} + \widehat{\bfH}^T\widehat{\bfR}^{-1}\widehat{\bfH}.
	\end{equation}
	The formulation of the objective function given by \eqref{eq:6:IncCostFn} is too expensive to be used in practice both in terms of computation, but also storage. The number of state variables, $N$, is very large and typically $\bfB$ cannot be stored explicitly.
	
	The control variable transform (CVT) formulates the objective function in terms of alternative `control variables', which means that the background matrix $\bfB$ does not need to be stored explicitly.    The CVT  is described in detail by Bannister \cite{bannister08,Bannister17}, and is often used in NWP applications. 

	The CVT may be applied to the incremental form of the variational problem \eqref{eq:6:IncCostFn}, via the change of variable $
	\delta\mathbf{ z}^{(l)}_0=	\bfB^{-1/2}\delta\bfx_0^{(l)}  $. This yields the objective function
	\begin{equation}\label{eq:6:PrecondCostFn}
	\widehat{J}(\delta\textbf{z}_0^{(l)})= \frac{1}{2}(\delta\textbf{z}_0^{(l)}-\delta\bfz_b^{(l)})^T (\delta\textbf{z}_0^{(l)}-\delta\bfz_b^{(l)})+ \frac{1}{2}\left(\widehat{\textbf{d}}^{(l)}- \widehat{\bfH}\bfB^{1/2}\delta\textbf{z}_0^{(l)}\right)^T\widehat{\bfR}^{-1}\left(\widehat{\textbf{d}}^{(l)}-\widehat{\bfH}\bfB^{1/2}\delta\textbf{z}_0^{(l)}\right),
	\end{equation}

	where  $\delta\bfz_b^{(l)} = \bfB^{-1/2}\delta\bfx_b^{(l)}$,
	and
	
	\begin{equation}
	\widehat{\textbf{d}}^{{(l)}^T}=\left[\bfd_o^{{(l)}^T},\bfd_1^{{(l)}^T},\dots,\bfd_n^{{(l)}^T}\right]
	\end{equation}
	is a vector made up of the innovation vectors.

	This yields a Hessian for the incremental 4D-Var problem with the CVT given by
	\begin{equation}\label{eq:6:Sh}
	\widehat{\bfS}=\bfI_N + \bfB^{1/2}\widehat{\bfH}^T\widehat{\bfR}^{-1}\widehat{\bfH}\bfB^{1/2}.
	\end{equation}
	Therefore using the CVT  is equivalent to pre- and post-multiplying the Hessian of the incremental data assimilation problem \eqref{eq:Hessianunprecond} by $\bfB^{1/2}$ (the uniquely defined, symmetric square root of $\bfB$). 
	{\color{black} The exact value of $\bfB^{-1/2}$ is not computed, but rather an approximation is constructed using physical and statistical knowledge of the system of interest \cite{bannister08}. }
	The CVT can be interpreted as preconditioning the Hessian by $\bfB^{1/2}$. The data assimilation formulation described in \eqref{eq:6:PrecondCostFn} is often referred to as the preconditioned data assimilation problem, and this naming convention will be used throughout the remainder of the paper.	
	We note that as we assume ${\bfB}$ and $\widehat{\bfR}$  are strictly positive definite, $\Sh$ is also symmetric positive definite. 

	The preconditioned Hessian \eqref{eq:6:Sh} highlights the computational benefit of using the CVT. 
	{\color{black} For most NWP applications there are} fewer observations than state variables (typically a difference of two orders of magnitude \cite{Carrassi18}), meaning that the second term in \eqref{eq:6:Sh} is rank deficient. Therefore the preconditioned Hessian is a low-rank update to the identity, and hence its minimum eigenvalue is unity. This guarantees that the preconditioned Hessian will not suffer from small minimum eigenvalues that often result in ill-conditioning for the unpreconditioned problem. This improved conditioning is expected to lead to faster convergence of the associated data assimilation algorithm. 
	
	In this paper we study the conditioning of the Hessian of the CVT objective function \eqref{eq:6:PrecondCostFn} as a proxy for convergence of the preconditioned data assimilation problem. We develop bounds on the condition number of \eqref{eq:6:Sh} in terms of its constituent matrices. Separating the contribution of each matrix in the bounds allows us to investigate the effect of changes to each component of the data assimilation system on conditioning and convergence. In particular, we focus on the introduction of  correlated observation error covariance matrices within the preconditioned framework.
	
	\subsection{{\color{black} Some inequalities for the eigenvalues of the product of  positive semidefinite Hermitian matrices}}	\label{sec:6:EvalTheory}
	
	For the remainder of this manuscript, we use the following order of eigenvalues: For a matrix $\bfA \in \mathbb{R}^{k\times k}$ let the eigenvalues $\lambda_i$ be such that  $\lambda_{max}(\bfA) = \lambda_{1}(\bfA) \ge \lambda_{2}(\bfA) \ge \dots \ge \lambda_k(\bfA) = \lambda_{min}(\bfA)$.

	In this section we introduce theoretical results from linear algebra. These will be used in Section \ref{sec:6:PrecondTheory} to develop new bounds on the condition number of the preconditioned Hessian in terms of its constituent matrices.   We also present existing bounds on the Hessian of  the preconditioned 3D-Var problem. {\color{black} The 3D-Var problem is obtained by setting $n=0$ in \eqref{eq:6:CostFn}.} In the numerical experiments of Section \ref{sec:6:Numerics}  we will compare these existing bounds with the new bounds developed in Section \ref{sec:6:PrecondTheory}.
	
	We begin by formally defining the condition number.
	\begin{mydef}\label{def:6:condS}
		\cite[Sec. 2.7.2]{golub96} For $\bfA \in \mathbb{R}^{k\times k}$ {\color{black} symmetric positive-definite we define the condition number $\kappa(\bfA)$ by 
			\begin{equation}
			\kappa(\bfA) = ||\bfA||||\bfA^{-1}||.
			\end{equation}}
		We then characterise the condition number in the $2-$norm of $\bfA$ as
		\begin{equation}\label{eq:condno}
		\kappa_2(\bfA) = ||\bfA||_2||\bfA^{-1}||_2 = \frac{\lambda_1(\bfA)}{\lambda_k(\bfA)},
		\end{equation}
		{\color{black} where $\lambda_i$ are the eigenvalues of $\bfA$. 
			The condition number in the 2-norm} shall be referred to as the condition number {\color{black} and denoted $\kappa(\bfA)$} for the remainder of this work.
	\end{mydef}
	We recall our additional assumption that both $\bfB$ and $\widehat{\bfR}$ are strictly positive definite. Since $\Sh$ is symmetric positive definite we apply the characterisation of the condition number given by \eqref{eq:condno} throughout this paper.

	We present two results which bound the eigenvalues of a matrix product in terms of the product of the eigenvalues of the individual matrices. These will be used in Section \ref{sec:6:PrecondTheory} to separate the contribution of the background and observation error covariance matrices to  $\kappa(\Sh)$.
	\begin{theorem} \label{theorem:6:lidskii} 
		Let $\bfF, \mathbf{G} \in \mathbb{C}^{d\times d}$ be positive semi-definite Hermitian matrices and {\color{black}let $i_1, \dots i_k$ denote an ordered  subset of the integers $\{1,\dots ,d\}$}. Then
		\begin{equation}
		\sum_{t=1}^{k}\lambda_{i_t}(\mathbf{F}\mathbf{G}) \le \sum_{t=1}^{k}\lambda_{i_t}(\mathbf{F}) \lambda_t(\mathbf{G}), \quad k=1,\dots,d-1.
		\end{equation}
	\end{theorem}
	
	\begin{proof}
		The proof is given by Wang and Zhang \cite{wang92}, Theorem 3.
	\end{proof}

	\begin{theorem}\label{theorem:6:wangsum}
		Let $\bfF, \mathbf{G} \in \mathbb{C}^{d\times d}$ be positive semidefinite Hermitian and {\color{black}let $i_1, \dots i_k$ denote an ordered  subset of the integers $\{1,\dots ,d\}$}. Then
		\begin{equation}
		\sum_{t=1}^{k} \lambda_{i_t}(\bfF\mathbf{G})\ge \sum_{t=1}^k \lambda_{i_t}(\bfF)\lambda_{d-t+1}(\mathbf{G}).
		\end{equation}
	\end{theorem}
	\begin{proof}
		The proof is given by Wang and Zhang \cite{wang92}, Theorem 4.
	\end{proof}

	These results will be used to develop bounds on the condition number of the Hessian \eqref{eq:6:Sh}.

	We now present an existing bound on the condition number of the 3D-Var preconditioned Hessian, $\kappa(\Sh)$, from Haben \cite{haben11c}.

	\begin{theorem}\label{thm:6:habenbound}
		Let $\bfB \in \mathbb{R}^{N\times N}$ be the background error covariance matrix and $\bfR \in \mathbb{R}^{p\times p}$ be the observation error covariance matrix with $p < N$. Then the following bounds are satisfied by the condition number of the preconditioned 3D-Var Hessian
		$\widehat{\mathbf{S}}=\mathbf{I}_{N}+\mathbf{B}^{1 / 2} \mathbf{H}^{T} \mathbf{R}^{-1} \mathbf{H B}^{1 / 2}$
		\begin{equation}\label{eq:6:habenbound}
		1+\frac{1}{p} \sum_{i, j=1}^{p}\left(\mathbf{R}^{-1 / 2} \mathbf{H B H}^{T} \mathbf{R}^{-1 / 2}\right)_{i, j} \leq \kappa(\widehat{\mathbf{S}}) \leq 1+\left\|\mathbf{R}^{-1 / 2} \mathbf{H} \mathbf{B} \mathbf{H}^{T} \mathbf{R}^{-1 / 2}\right\|_{\infty}.
		\end{equation}
	\end{theorem}
	\begin{proof}
		The proof is given by Haben \cite[]{haben11c}, Theorem 6.2.1.
	\end{proof}
	We note that the result of Theorem \ref{thm:6:habenbound} 	extends naturally to the 4D-Var problem by replacing $\bfR$ with $\widehat{\bfR}$ and $\bfH$ with $\widehat{\bfH}$. 
	As discussed at the end of Section \ref{sec:6:DA}, we want to develop bounds that separate the contribution of each constituent matrix.  This will allow us to study how altering a single term, particularly the observation error covariance matrix, is likely to affect the conditioning and convergence of the preconditioned data assimilation system. 
	As well as being interesting from a theoretical perspective, improved understanding of the influence of individual terms will be useful for practical applications.
	For example, when introducing new observation operators or observation error covariance matrices into operational systems, bounds which separate the role of each matrix will provide insight into  
	how the conditioning of the preconditioned 4D-Var problem is likely to change.
	However, as the bounds given by \eqref{eq:6:habenbound} do not separate out each term, they are likely to be tighter than the new bounds which are presented in Section \ref{sec:6:PrecondTheory}.
	In Section \ref{sec:6:Numerics} we will numerically compare the bounds  given by \eqref{eq:6:habenbound} with those developed in Section \ref{sec:6:PrecondTheory}.

	\section{Theoretical bounds on the Hessian of the preconditioned problem}\label{sec:6:PrecondTheory}
In this section we develop new theoretical bounds on the condition number of the Hessian of the preconditioned variational data assimilation problem, following similar methods to the unpreconditioned case in Tabeart et al. \cite{tabeart17a}. These bounds will all be presented in terms of the Hessian of the {\color{black}preconditioned} 4D-Var problem \eqref{eq:6:Sh}. For the case $n=0$, $\bfRh \equiv \bfR_0\in \mathbb{R}^{p_0\times p_0}$ and $\bfHh \equiv \bfH_0 \in \mathbb{R}^{p_0 \times N}$, meaning that the bounds in this Section will also apply directly to the preconditioned 3D-Var Hessian. This relation will be used in the numerical experiments presented in Section \ref{sec:6:Numerics}. 
	{\color{black}
			
		\begin{mainass}\label{ass:6:4DVar}
			The total number of observations across the time window, $p$, is smaller than the number of state variables, i.e. $p<N$.
		\end{mainass}}

	The first result shows that the condition number can be calculated via the eigenvalues of the rank-$p$ update $\bfB^{1/2}\bfHh^T\bfRh^{-1}\bfHh\bfB^{1/2}$.
	
	\begin{lemma}\label{lem:6:kShinevals}
		Following the Key Assumption  we can express the condition number of $\Sh$ as
		\begin{align}
		\kappa(\Sh) &=1+  \lambda_1(\bfB\bfHh^T\bfRh^{-1}\bfHh) \label{eq:6:kShinevalsa}\\
		&=1+\lambda_1(\bfRh^{-1}\bfHh\bfB\bfHh^T).\label{eq:6:kShinevalsb}
		\end{align}
	\end{lemma}
	\begin{proof}
		We begin by showing that $\kappa(\Sh) = 1+ \lambda_{1}(\bfB^{1/2}\bfHh^T\bfRh^{-1}\bfHh\bfB^{1/2})$, as was presented in Haben \cite[]{haben11c}, Equation (4.2).
		We define $\bfB^{1/2}\bfHh^T\bfRh^{-1}\bfHh\bfB^{1/2} = \bfC$ and write $\Sh = \bfI+\bfC$. Let $\lambda_1\ge\lambda_2\ge\dots \ge \lambda_N$ be the eigenvalues of $\bfC$, with corresponding eigenvectors $v_i$. As $p<N$, $\bfC$ is rank-deficient and therefore $\lambda_N=0$.

		Therefore $\lambda_{N}(\Sh) = 1$, and $\kappa(\Sh) = \lambda_{1}(\Sh) = 1+\lambda_1(\bfC)$. 
	{\color{black} Matrices $\bfA\bfB$ and $\bfB\bfA$ have the same non-zero eigenvalues \cite{harville97}, and therefore we can write
			\begin{equation}
			\lambda_1(\bfC) = \lambda_1(\bfB^{1/2}\bfHh^T\bfRh^{-1}\bfHh\bfB^{1/2} ) = \lambda_1(\bfB\bfHh^T\bfRh^{-1}\bfHh) = \lambda_1(\bfRh^{-1}\bfHh\bfB\bfHh^T). 
			\end{equation}
			Hence, we obtain the result
			\begin{equation}
			\kappa(\Sh) = 1+\lambda_{1}(\bfB\bfHh^T\bfRh^{-1}\bfHh) = 1+ \lambda_1(\bfRh^{-1}\bfHh\bfB\bfHh^T).
			\end{equation} }
	\end{proof}			
	The result of Lemma \ref{lem:6:kShinevals} shows that computing $\kappa(\Sh)$ only requires the computation of the maximum eigenvalue of a single matrix product. We also note that the matrix products that appear in \eqref{eq:6:kShinevalsa} and \eqref{eq:6:kShinevalsb} are of different dimensions: $\bfB\bfHh^T\bfRh^{-1}\bfHh \in \mathbb{R}^{N\times N}$ and $\bfRh^{-1}\bfHh\bfB\bfHh^T\in\mathbb{R}^{p\times p}$.
	Additionally, by the Key Assumption (as $p<N$) the first matrix product is always rank deficient, whereas for the case that $\bfHh^T\bfRh^{-1}\bfHh$ is rank $p$, the second matrix product is full rank.
	
	Previous studies \cite{haben11c,tabeart17a} have considered the effect of separately changing the variances and correlations associated with the background  and observation covariances.  In our numerical experiments in Section \ref{sec:6:Numerics}, we will focus on the role of the correlations in $\bfB$ and $\bfR$ in the conditioning of the preconditioned assimilation problem and assume the variances are constant, {\color{black} i.e. $\bfB = \sigma_B^2\widetilde{\bfB}$, $\bfR = \sigma_{R_i}^2\widetilde{\bfR_i}$, where $\sigma_{R_i},\sigma_B \in \mathbb{R}$}.  In that case it is known \cite{haben11b} that  $\kappa(\Sh)$  increases and decreases with the ratio of the background variance to the observation variance.  We therefore assume in the experiments that the variances all take unit values and examine how changes to the background and observation correlations affect the conditioning.

	\subsection{General bounds on the condition number}\label{sec:6:genbounds}
	We now develop bounds on the condition number of $\Sh$ in terms of its constituent matrices. {\color{black}We assume that $\bfB$ and $\widehat{\bfR}$ are strictly positive definite, and that the Key Assumption holds. Otherwise we make no further restrictions on the structure of the constituent matrices in this Section.}
	
	\begin{theorem}\label{cor:6:BnormHRH}
	Given the Key Assumption we can bound $\kappa(\widehat{\bfS}) = \kappa(\textbf{I}_N + \bfB^{1/2}\bfHh^T\bfRh^{-1}\bfHh\bfB^{1/2})$ by 
		\begin{equation}\label{eq:6:BnormHRH}
		\begin{split}
		&	1+ \max \left\{\lambda_1(\bfHh^T\bfRh^{-1}\bfHh)\lambda_N(\bfB), \quad \frac{\lambda_1(\bfHh\bfB\bfHh^T)}{\lambda_1(\bfRh)},\quad\frac{ \lambda_p(\bfHh\bfB\bfHh^T)}{\lambda_p(\bfRh)} \right\} \\
		&\le \kappa(\widehat{\bfS})  \le 1+ \min\left\{ \lambda_{1}(\bfB)\lambda_{1}(\bfHh^T\bfRh^{-1}\bfHh), \quad \frac{\lambda_{1}(\bfHh\bfB\bfHh^T)}{\lambda_p(\bfRh)}\right\}.
		\end{split}
		\end{equation}
	\end{theorem}
	\begin{proof}
		We write $\kappa(\Sh)$ as in the statement of Lemma \ref{lem:6:kShinevals}.
		To obtain the upper bound of  \eqref{eq:6:BnormHRH}, we use the result of Theorem \ref{theorem:6:lidskii} to separate the contribution of the background and observation term 
		\begin{equation}\label{eq:bound1Thm7}
		\begin{split}
		\kappa(\Sh) &= 1 + \lambda_1(\bfB\bfHh^T\bfRh^{-1}\bfHh) \le	 1+ \lambda_1(\bfB)\lambda_1(\bfHh^T\bfRh^{-1}\bfHh).\\
		\end{split}
		\end{equation}
		Similarly the alternative formulation from Lemma \ref{lem:6:kShinevals} yields
		\begin{equation}
		\begin{split}
		\kappa(\Sh) &= 1 + \lambda_{1}(\bfRh^{-1}\bfHh\bfB\bfHh^T) \\
		&\le
		1+ \frac{1}{\lambda_{p}(\bfRh)}\lambda_{1}(\bfHh\bfB\bfHh^T).
		\end{split}
		\end{equation}

		Combining these two expressions yields the upper bound in the theorem statement.
		
		To compute the lower bound of \eqref{eq:6:BnormHRH}, we apply the result of Theorem \ref{theorem:6:wangsum} to  \eqref{eq:6:kShinevalsa}  with $k=1,i_1=1,d=N$. This yields
		\begin{equation}
		\begin{split}
		\lambda_1(\bfB\bfHh^T\bfRh^{-1}\bfHh)  &\ge \max\left\{ \lambda_1(\bfHh^T\bfRh^{-1}\bfHh) \lambda_N(\bfB), \quad \lambda_N(\bfHh^T\bfRh^{-1}\bfHh) \lambda_1(\bfB) \right\} \\
		& \ge \lambda_1(\bfHh^T\bfRh^{-1}\bfHh) \lambda_N(\bfB).
		\end{split}
		\end{equation}
		This last inequality is due to the fact that $\bfHh^T\bfRh^{-1}\bfHh$ is rank-deficient.
		It follows from Fact 5.11.14 of Bernstein \cite{Bernstein} that 	$\lambda_{i}(\bfRh^{-1}) = \frac{1}{\lambda_{p-i+1}(\bfRh)}$. 
		Applying the result of Theorem \ref{theorem:6:wangsum} to \eqref{eq:6:kShinevalsb}  with $k=1,i_1=1,d=N$ we obtain 
		\begin{equation}\label{eq:bound4Thm7}
		\lambda_1(\bfRh^{-1}\bfHh\bfB\bfHh^T)  	
		\ge \max \left\{\frac{ \lambda_1(\bfHh\bfB\bfHh^T)}{\lambda_1(\bfRh)},\quad \frac{ \lambda_p(\bfHh\bfB\bfHh^T)}{\lambda_p(\bfRh)} \right\}.
		\end{equation}
		Combining the results of \eqref{eq:bound1Thm7}-\eqref{eq:bound4Thm7} yields \eqref{eq:6:BnormHRH} as required.
		
	\end{proof}
	
	We can separate the contribution of the observation error covariance matrix from the observation operator to give the following bound.
	
	\begin{cor} \label{lemma:6:factorisedprecond}
		Under the same conditions as in Theorem \ref{cor:6:BnormHRH}, we can bound $\kappa(\widehat{\bfS})$  by
		
		\begin{equation} \label{eq:6:factoredprecond}
		\begin{split}
		1+ \max\left\{\frac{\lambda_{p}(\bfHh\bfHh^T)\lambda_N(\bfB)}{\lambda_p(\bfRh)}, \frac{\lambda_{1}(\bfHh\bfHh^T)\lambda_N(\bfB)}{\lambda_1(\bfRh)} \right\}\le \kappa(\widehat{\bfS})\le 1 + \frac{\lambda_{1}(\bfB)}{\lambda_{p}(\bfRh)}\lambda_{1}(\bfHh\bfHh^T)
		\end{split}
		\end{equation}
	\end{cor}
	\begin{proof}
	We begin by considering the upper bound of \eqref{eq:6:BnormHRH}. 
		By Theorem 21.10 of Harville \cite{harville97},  $\bfHh^T\bfRh^{-1}\bfHh$ has precisely the same nonzero eigenvalues as $\bfRh^{-1} \bfHh\bfHh^T$. It follows from Fact 5.11.14 of Bernstein \cite{Bernstein} that 	$\lambda_{i}(\bfRh^{-1}) = \frac{1}{\lambda_{p-i+1}(\bfRh)}$.  Applying Theorem \ref{theorem:6:lidskii} 
		for $k=1, i_1=1, d=p$  to $	\lambda_{1}(\bfRh^{-1}\bfHh\bfHh^T)$ yields:
		\begin{align}\label{eq:6:lmaxofproduct}
		\lambda_{1}(\bfRh^{-1}\bfHh\bfHh^T) \le 
		\frac{\lambda_{1}(\bfHh\bfHh^T)}{\lambda_{p}(\bfRh)}.
		\end{align}
		By Theorem 21.10 of Harville \cite{harville97}, $\bfHh\bfB\bfHh^T$ has precisely the same nonzero eigenvalues as $\bfB\bfHh^T\bfHh$. Applying Theorem \ref{theorem:6:lidskii} for $k=1, i_1=1, d=N$ yields:
		\begin{align}\label{eq:6:lmaxofproductB}
		\lambda_{1}(\bfB\bfHh^T\bfHh) \le \lambda_{1}(\bfB)\lambda_{1}(\bfHh^T\bfHh) =  \lambda_{1}(\bfB)\lambda_{1}(\bfHh\bfHh^T).
		\end{align}
		The final equality arises as the nonzero eigenvalues of $\bfHh\bfHh^T$ are equal to those of $\bfHh^T\bfHh$. 
		Therefore the two cases from Theorem \ref{cor:6:BnormHRH} yield the same `factorised' upper bound, and gives the upper bound in \eqref{eq:6:factoredprecond}.

		We now consider the first term in the lower bound of \eqref{eq:6:BnormHRH} and bound
		$\lambda_{1}(\bfHh^T\bfRh^{-1}\bfHh)$ below. We separate the contribution of $\bfRh$ and $\bfHh\bfHh^T$ using Theorem \ref{theorem:6:wangsum} for $k=1, i_1=1, d=p$. This yields
		\begin{align}
		\lambda_{1}(\bfRh^{-1} \bfHh\bfHh^T)
		\ge \max \left\{  \frac{\lambda_{1}(\bfHh\bfHh^T)}{\lambda_{1}(\bfRh)},\frac{\lambda_{p}(\bfHh\bfHh^T)}{\lambda_{p}(\bfRh)}\right\}.
		\end{align}
		Multiplying this by $\lambda_N(\bfB)$ gives 
		the two terms that appear in the lower bound of \eqref{eq:6:factoredprecond}. 
		
		We now consider the second term of \eqref{eq:6:BnormHRH} and bound $\lambda_{1}(\bfHh\bfB\bfHh^T)$ below. We separate the contribution of $\bfB$ and $\bfHh^T\bfHh$ using Theorem \ref{theorem:6:wangsum} for $k=1, i_1=1, d=N$. This yields
		\begin{equation}
		\begin{split}
		\lambda_{1}(\bfB\bfHh^T\bfHh) &\ge \max \left\{  \lambda_{1}(\bfB)\lambda_N(\bfHh^T\bfHh),  \lambda_{N}(\bfB)\lambda_1(\bfHh^T\bfHh) \right\}\\& \ge \lambda_{N}(\bfB)\lambda_1(\bfHh^T\bfHh).
		\end{split}
		\end{equation}
		The last inequality follows as $\bfHh^T\bfHh$ is not full rank and therefore $\lambda_N(\bfHh^T\bfHh) = 0$. Multiplying this result by $1/\lambda_1(\bfRh)$ gives the same value as the second term in \eqref{eq:6:factoredprecond}.

		Finally, we bound the third term of the lower bound in \eqref{eq:6:BnormHRH}. By Theorem 21.10 of Harville \cite{harville97}, $\lambda_{p}(\bfHh\bfB\bfHh^T) = \lambda_{p}(\bfB\bfHh^T\bfHh)$.  
		Applying Theorem \ref{theorem:6:wangsum} for $k=1, i_1=p, d=N$ yields
		\begin{equation}\label{eq:6:temp2}
		\begin{split}
		\lambda_{p}(\bfB\bfHh^T\bfHh) &\ge \max  \{ \lambda_{p}(\bfB)\lambda_{N}(\bfHh^T\bfHh),\lambda_{N}(\bfB)\lambda_p(\bfHh^T\bfHh)\}\\
		& \ge \lambda_{N}(\bfB)\lambda_p(\bfHh^T\bfHh).
		\end{split}
		\end{equation}
		Multiplying the second term of \eqref{eq:6:temp2} by $1/\lambda_{p}(\bfRh)$ gives the first term in \eqref{eq:6:factoredprecond}, as $\lambda_p(\bfHh^T\bfHh)= \lambda_p(\bfHh\bfHh^T)$. 
	\end{proof}
	In general it is not possible to determine which term in the lower bound of \eqref{eq:6:factoredprecond} is larger, as this will depend on the choice of $\bfB, \bfHh$ and $\bfRh$. However, we are able to comment on how the bounds are likely to be altered by changes to individual matrices. 
	As we increase $\lambda_p(\bfRh)$ both the upper bound and first term in the lower bound decrease. 
	Increasing $\lambda_{1}(\bfRh)$ will lead to a decrease in the second term of the lower bound.
	As $\lambda_{N}(\bfB)$ increases, the lower bound will increase but the upper bound will remain unchanged. 
	Increasing $\lambda_{1}(\bfB)$ will yield a larger upper bound and has no effect on the lower bound. 
	Larger values of $\lambda_{p}(\bfHh\bfHh^T)$ will lead to increases to the first term in the lower bound and larger values of $\lambda_{1}(\bfHh\bfHh^T)$ will lead to increases of the upper bound and second term of the lower bound.
	In the experiments in Section \ref{sec:6:Numerics} we will study how each of these terms change with interacting parameters, and assess which lower bound is tighter for a variety of situations.

	\subsection{Bounds on the condition number in the case of circulant error covariance matrices}	\label{sec:6:circulant}
	The theoretical bounds presented in Section \ref{sec:6:genbounds} apply for any choice of observation and background error covariance matrices. However, for a given numerical framework, general bounds can typically be improved by exploiting specific structure of the matrices being used \cite{haben11c}. In this section we will show that under additional assumptions on the structure of the error covariance matrices and observation operator, the bounds given by \eqref{eq:6:habenbound} yield the exact value of $\kappa(\widehat{\bfS})$.

	We begin by defining circulant matrices. Circulant matrices are a natural choice for spatial correlation matrices on a one-dimensional periodic domain, as they yield correlation matrices that are homogeneous and isotropic \cite{haben11c}. We will make use of this structure in the numerical experiments presented in Section \ref{sec:6:Numerics}. \linebreak

	\begin{mydef}\cite{davis79}\label{def:6:Circulant}
		A circulant matrix $\bfC \in \mathbb{R}^{N\times N}$ is a matrix of the form 
		\begin{center}
			\begin{math}
			\bfC =
			\begin{pmatrix}
			c_0 & c_1 & c_2 & \cdots & c_{N-2} & c_{N-1} \\
			c_{N-1} & c_0 & c_1 & \cdots & c_{N-3} & c_{N-2} \\
			c_{N-2} & c_{N-1} & c_0 & \cdots &c_{N-4} & c_{N-3} \\
			\vdots& \vdots & \vdots & \ddots &\vdots &\vdots \\
			c_2& c_3 & c_4 & \cdots & c_0 & c_1 \\
			c_1 & c_2 & c_3 & \cdots & c_{N-1} & c_0
			\end{pmatrix}.
			\end{math}
		\end{center}
		
	\end{mydef}
	One computationally beneficial property of circulant matrices is that their eigenvalues  can  be calculated directly via a discrete Fourier transform. 	As we shall see in Theorem \ref{thm:6:CirculantEvalsvecs}, any circulant matrix of dimension $N$ admits the same eigenvectors. 
	
	\begin{theorem} \label{thm:6:CirculantEvalsvecs}
		The eigenvalues of a circulant matrix $\bfC\in \mathbb{R}^{N\times N}$, as given by Definition~\ref{def:6:Circulant}, are given by
		\begin{equation} \label{eq:6:CirculantEvals}
		\gamma_m=\sum_{k=0}^{N-1} c_k \omega^{mk},
		\end{equation} 
		with corresponding eigenvectors
		\begin{equation}\label{eq:6:CirculantEvecs}
		\mathbf{v}_{m}=\frac{1}{\sqrt{N}}(1,\omega^m, \cdots, \omega^{m(N-1)}),
		\end{equation}
		where $\omega=e^{-2\pi i/N}$ is an $N-$th root of unity.
	\end{theorem}
	\begin{proof}
		See \cite{gray06} for full derivation.
	\end{proof}

	Our numerical experiments in Section \ref{sec:6:Numerics} will use circulant background and observation error covariance matrices. 
	When both $\bfHh\bfB\bfHh^T$ and $\bfRh$ are circulant, with some additional assumptions on the entries of matrix products, we can prove that the upper and lower bounds given by Theorem \ref{thm:6:habenbound} are equal and yield the exact value of $\kappa(\widehat{\bfS})$.
	\begin{cor} \label{lem:6:circulantequivbound}
		If $\bfHh\bfB\bfHh^T\in \mathbb{R}^{p\times p}$ and $\bfRh\in\mathbb{R}^{p\times p}$ are circulant matrices, and all of the entries of $\bfRh^{-1/2}\bfHh\bfB\bfHh^T\bfRh^{-1/2}$  are positive, then the upper and lower bounds in Theorem \ref{thm:6:habenbound} are equal, and the bound on $\kappa(\widehat{\bfS})$ is exact.
	\end{cor}
	\begin{proof}
		The product of circulant matrices is a circulant matrix, the inverse of a circulant matrix is circulant \cite{gray06}, and the square root of a circulant matrix is also circulant \cite{mei12}. Therefore if the product $\bfHh\bfB\bfHh^T$ is circulant then the product $\bfRh^{-1/2}\bfHh\bfB\bfHh^T\bfRh^{-1/2}$ is circulant, as $\bfRh$ is circulant by assumption of the corollary.
		
		The lower bound of \eqref{eq:6:habenbound} computes the average row sum of the product $\bfRh^{-1/2}\bfHh\bfB\bfHh^T\bfRh^{-1/2}$. 
			As the product is circulant, each row has the same sum, given by $\sum_{k=0}^{p-1} c_k$, where $c_i$ is the ith entry of the first row of the circulant matrix (as introduced in Definition \ref{def:6:Circulant}).
		
		The upper bound of \eqref{eq:6:habenbound} returns the maximum absolute row sum of the product. As the product is circulant with only positive entries, all absolute row sums are identically equal to  $\sum_{k=0}^{p-1} |c_k| = \sum_{k=0}^{p-1} c_k $. Hence, we have equality of lower and upper bounds and hence the exact value for $\kappa(\widehat{\bfS})$.
	\end{proof} 
	This result shows that, if the additional assumptions are satisfied, we can compute $\kappa(\widehat{\bfS})$ directly using  \eqref{eq:6:habenbound}. 
	{\color{black}If $\bfB$ and $\bfRh$ are both circulant, and the observed state variables are regularly spaced then the first assumption of Corollary~\ref{lem:6:circulantequivbound} is satisfied. The requirement that all entries of $\bfRh^{-1/2}\bfHh\bfB\bfHh^T\bfRh^{-1/2}$ are positive is less straightforward to guarantee a priori, and depends on the specific structure of the three matrices being considered. In particular, even if all of the entries of $\bfRh, \bfHh$ and $\bfB$ are positive, entries of the product $\bfRh^{-1/2}\bfHh\bfB\bfHh^T\bfRh^{-1/2}$ can still be negative. The result of Corollary \ref{lem:6:circulantequivbound} will be used in the numerical experiments in the next section to compare the performance of the new bounds given by \eqref{eq:6:factoredprecond} and the existing bounds given by Theorem \ref{thm:6:habenbound}.}

	\section{Numerical framework}\label{sec:NumFramework}
	In this section we  describe  the numerical framework that will be used to study how the bounds on the preconditioned Hessian  \eqref{eq:6:Sh} compare with the actual value of $\kappa(\Sh)$. {\color{black}Although the bounds that were developed in Section \ref{sec:6:PrecondTheory} were developed for the 4D-Var problem, the numerical experiments presented in Section \ref{sec:6:Numerics} will be conducted for a 3D-Var problem. This allows us to use the framework that was introduced in Tabeart et al. \cite{tabeart17a}, and directly compare the preconditioned and unpreconditioned formulations in the same numerical setting.}
	We note that in the case of 3D-Var, $\bfRh$ and $\bfHh$ simplify to the standard observation error covariance matrix $\bfR$ and observation operator $\bfH$, respectively in \eqref{eq:6:PrecondCostFn}, \eqref{eq:6:Sh} and all the bounds in Section \ref{sec:6:PrecondTheory}. The Hessian that is used for the experiments in this section is therefore given by $\Sh = \bfI_N + \bfB^{1/2}\bfH^T\bfR^{-1}\bfH\bfB^{1/2}$.

	We now define the different components of the numerical framework. Our domain is the unit circle, and we fix the ratio of the number of state variables to observations as $N=2p$, i.e. twice as many state variables as observations. 	
	Similarly to Tabeart et al. \cite{tabeart17a} we define both the observation and background error covariance matrices to have a circulant structure with unit variances. Circulant matrices are a natural choice for correlations on a periodic domain with evenly distributed state variables. They also admit useful theoretical properties as was discussed in Section \ref{sec:6:circulant}. 
	The use of circulant error covariance matrices allow us to better understand the interaction between different terms in the Hessian, and to isolate the impact of parameter changes.

	The experiments presented in this paper will use circulant matrices arising from the second order autoregressive (SOAR)  
	correlation function \cite{Daley92,johnson03}.  SOAR matrices are used in NWP applications as a horizontal correlation function \cite{simonin18}
	and are fully defined by a correlation lengthscale for a given domain. We remark that we substitute the great circle distance in the SOAR correlation function with the chordal distance \cite{gaspari99,jeong15} to ensure that the properties of positive definiteness are satisfied and that we obtain a valid correlation matrix. 
	
	\begin{mydef}\label{def:6:SOARMatrix}
		The SOAR error correlation matrix on the unit circle is given by 
		\begin{equation} \label{eq:6:SOAReq}
		\bfD(i,j) = \Bigg(1+ \frac{\Big|2 \sin\Big(\frac{\theta_{i,j}}{2}\Big)\Big|}{L}\Bigg) \exp \Bigg(\frac{-\Big|2 \sin\Big(\frac{\theta_{i,j}}{2}\Big)\Big|}{L}\Bigg),
		\end{equation}
		where $L>0$ is the correlation lengthscale and $\theta_{i,j}$  denotes the angle between grid points $i$ and $j$. The chordal distance between adjacent grid points is given by
		\begin{equation}\label{eq:6:Deltax}
		\Delta x = 2 \sin\Big(\frac{\theta}{2}\Big) = 2 \sin\Big(\frac{\pi}{N}\Big),
		\end{equation} 
		where $N$ is the number of gridpoints and $\theta = \frac{2\pi}{N} $ is the angle between adjacent gridpoints.
	\end{mydef}

	Both the background and observation error covariance matrices for the experiments presented in Section \ref{sec:6:Numerics} will be SOAR with constant unit variance. We will denote their respective lengthscales by $L_B$ and $L_R$.

	We now introduce the observation operators that will be used for the 3D-Var experiments. Three of our observation operators are the same as those used in Tabeart et al. \cite{tabeart17a} which we state again for clarity. 
	
	\begin{mydef}\label{def:6:obsops}
		The observation operators $\bfH_1$, $\bfH_2$, $\bfH_3 \in \mathbb{R}^{p \times N}$, for $N=2p$, are defined as follows:
		\begin{align} 
		\bfH_1 (i,j)&=
		\begin{cases}
		1, & j=i \text{ for } i=1, \dots , p \\
		0, & \text{otherwise}.
		\end{cases}\\ 
		\bfH_2 (i,j)&=
		\begin{cases}
		1, & j=2i \text{ for } i=1, \dots , p \\
		0, & \text{otherwise}.
		\end{cases}\\
		\bfH_3 (i,j)&=
		\begin{cases}
		\frac{1}{5}, & j\in\{2i-2, 2i-1, 2i, 2i+1 ,2i+2 \pmod{N} \} \text{ for } i=1, \dots , p \\
		0, & \text{otherwise}.
		\end{cases}
		\end{align}
		
	\end{mydef}
	The first choice of observation operator, $\bfH_1$ corresponds to direct observations of the first half of the domain. The second observation operator, $\bfH_2$, corresponds to direct observations of alternate state variables. The third observation operator, $\bfH_3$, is a smoothed version of $\bfH_2$. Observations of alternate state variables are smoothed equally over 5 adjacent state variables. 
	The fourth choice of observation operator, $\bfH_4$ selects $p$ random direct observations. We considered a number of choices of random observation operator, and all choices yielded similar numerical results. In order to ensure a fair comparison, we fix the same choice of $\bfH_4$ for all of the results presented in Section~\ref{sec:6:Numerics}. This choice of observation operator is shown in Figure \ref{fig:6:H4}. Observations are spread over the whole domain, but are clustered rather than evenly distributed. Figure 2 of Tabeart et al. \cite{tabeart17a} shows a representation of a low dimensional version of the observation operator structure for $\bfH_1$, $\bfH_2$ and $\bfH_3$. For the numerical experiments in Section \ref{sec:6:Numerics}, we will use $p=100$ observations and $N=200$ state variables. In the unpreconditioned case, structure in the observation operator, such as regularly spaced observations, was important for the tightness of bounds and convergence of a conjugate gradient method \cite{tabeart17a}.  We therefore consider $\bfH_4$ as an operator without strict structure. This will allow us to see how structure (or the lack of it) affects the preconditioned problem. 
	
	\begin{figure}
		\includegraphics[width=\textwidth,clip,trim=0mm 7mm 0mm 0mm]{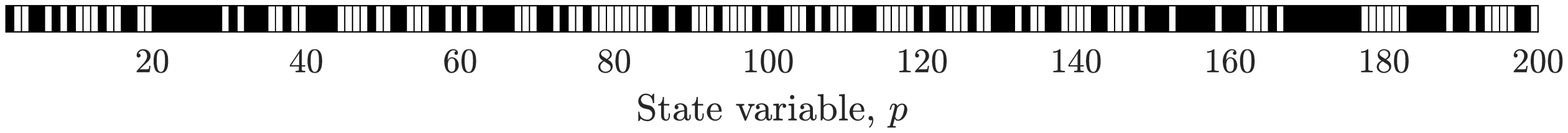}
		\caption[Representation of observation operator $\bfH_4$]{Representation of state variables that are observed for $\bfH_4$. Black denotes state variables that are observed directly and white denotes state variables that are not observed. \label{fig:6:H4}}
	\end{figure}

	\subsection{Changes to the condition number of the Hessian}\label{sec:6:condnodesign}
	Our first set of experiments consider how different combinations of parameters will alter the value of $\kappa(\Sh)$ and the bounds given by \eqref{eq:6:factoredprecond}. 
	We compute the condition number of the Hessian \eqref{eq:6:Sh} using the Matlab 2018b function $cond$ \cite{MATLAB:2018} and compare against the values given by our bounds. Table 1 (reproduced from Tabeart et al. 2018 \cite{tabeart17a}) shows that increasing the lengthscale of a SOAR correlation matrix will reduce its smallest eigenvalue and increase its largest eigenvalue.   The maximum and minimum eigenvalues of both error covariance matrices appear  in \eqref{eq:6:factoredprecond}. We can therefore predict how the bounds will change with varying parameter values.  	
	
	\begin{table}
		\caption{Reproduction of Table 1 from \cite{tabeart17a}. Summary of {\color{black} changes to the} eigenvalues of $\bfB\in\mathbb{R}^{200\times200}$ and $\bfR\in\mathbb{R}^{100\times100}$  with the lengthscales $L_B$ and $L_R$ for $\bfB$ and $\bfR$ both SOAR matrices.}
		\centering
		\label{table:2}
		\begin{tabular}{ l c c c c c}
			& \multicolumn{3}{c}{Lengthscale $L_R \ $ or $\ L_B$} \\
			
			& $0.1$ & $0.33 $ & $0.66$ & $0.99 $ &$1$ \\
			
			$\lambda_{N}(\bfR)$& $1.92\times 10^{-2}$&$ 5.74\times 10^{-4}$ & $7.21 \times10^{-5}$ & $2.14\times10^{-5}$ &$2.08 \times10^{-5} $\\
			$\lambda_{1}(\bfR)$& $6.40\times 10^{0}$&$ 2.26\times 10^{1}$ & $4.67 \times10^{1}$ & $6.36\times10^{1}$ &$6.40 \times10^{1} $\\
			$\lambda_{N}(\bfB)$&$2.54\times 10^{-3}$& $7.19\times10^{-5}$ & $8.99\times 10^{-6}$& $2.67\times10^{-6}$&$2.59\times 10^{-6}$\\
			$\lambda_{1}(\bfB)$&$1.28\times 10^1$&$4.51\times10^{1}$ & $9.35\times10^{1} $& $1.27\times10^{2}$ &$1.28\times 10^2$\\

		\end{tabular}
	\end{table}
	
	\begin{itemize}
		\item As $L_R$ (the lengthscale of the correlation function used to construct $\bfR$) increases, $\lambda_{p}(\bfR)$ decreases. This means that both the upper bound  and the first term in the lower bound of \eqref{eq:6:factoredprecond}  will increase. However, $\lambda_{1}(\bfR)$ increases with $L_R$ meaning that the second term in the lower bound will decrease. It is therefore not possible to determine whether the lower bound will increase or decrease with increasing $L_R$ in general. \item For the case of direct observations, all eigenvalues of $\bfH\bfH^T$ are equal to unity \cite{haben11c}. Therefore the first term in the lower bound of  \eqref{eq:6:factoredprecond} will always be greater than the second term. Both $\bfH_1$ and $\bfH_2$ correspond to direct observations; hence for these choices of observation operator the first term in the lower bound of \eqref{eq:6:factoredprecond} is the lower bound of $\kappa(\Sh)$.
		\item As $L_B$ (the lengthscale of the correlation function used to construct $\bfB$) increases, $\lambda_{1}(\bfB)$ increases. This means that the upper bound of \eqref{eq:6:factoredprecond} will increase with $L_B$. As $L_B$ increases, $\lambda_{N}(\bfB)$ decreases. This means that both terms in the  lower bound of \eqref{eq:6:factoredprecond} will decrease with increasing $L_B$.  Hence, the bounds \eqref{eq:6:factoredprecond} will diverge as $L_B$ increases.
	\end{itemize}
	
	We wish to assess whether the qualitative behaviour of $\kappa(\Sh)$ agrees with the qualitative behaviour of the bounds for our experimental framework. Additionally, we are interested in determining which term in the lower bound of \eqref{eq:6:factoredprecond} is largest, and whether this depends on the choice of $\bfB$, $\bfR$ and $\bfH$.
	
	In Section \ref{sec:6:Numerics} we compare the bounds given by \eqref{eq:6:factoredprecond} with those of \eqref{eq:6:habenbound}. As discussed at the end of Section \ref{sec:6:Background}, although we expect the bounds given by \eqref{eq:6:habenbound} to be tighter in many cases, separating the contribution of constituent matrices by using \eqref{eq:6:factoredprecond} will be qualitatively informative. 
	\subsection{Convergence of a conjugate gradient algorithm}\label{sec:6:CGdesign}
	Although conditioning of a problem is often used as a proxy to study convergence, there are well-known situations where the condition number provides a pessimistic indication of convergence speed{\color{black}, notably in the case of repeated or clustered eigenvalues (e.g Theorems 38.3, 38.5 of Trefethen and Bau \cite{TrefethenLloydN.LloydNicholas1997Nla}, and Theorem 38.4 of Gill, Murray and Wright \cite{Gill86}). We therefore wish to investigate how well the condition number of the preconditioned system reflects the convergence of a conjugate gradient method for our experimental framework. }
	Following a similar method to Section 5.3.2. of Tabeart et al. \cite{tabeart17a} we study how the speed of convergence of a conjugate gradient method applied to the linear system $\Sh\bfx = \mathbf{b}$ changes with the parameters of the system. We define $\bfx$ as a vector with features at a variety of scales, and then calculate $ \mathbf{b}=\Sh\bfx$ before recovering $\bfx$. We use the Matlab 2018b routine $pcg.m$ to recover $\bfx$ using the conjugate gradient method. As we are studying a preconditioned system, convergence is fast. In order to make the differences between parameter choices more evident we use a tolerance of $1\times 10^{-10}$ on the relative residual.  In Section \ref{sec:6:Numerics} we show results for one particular realisation of $\bfx$ to enable a fair comparison between different choices of $\bfR$, $\bfB$ and $\bfH$.
	A number of other values of $\bfx$ were tested, with similar results.
	
	We consider how changes to lengthscale and observation operator alter the convergence of the conjugate gradient method. For cases where convergence behaves differently to conditioning, we  study the spectrum of $\Sh$ to understand why these differences occur.

	\section{3D-Var experiments}\label{sec:6:Numerics}
	In this section we present the results of our numerical experiments. 
	Figures will be plotted as a function of  changes to correlation lengthscales for both $\bfB$ and $\bfR$. We recall that increasing the lengthscale of a SOAR correlation matrix will reduce its smallest eigenvalue and increase its largest eigenvalue \cite{waller16,tabeart17a}.  
	
	\begin{figure}
		\includegraphics[trim=15mm 0mm 15mm 0mm, clip, width=\textwidth]{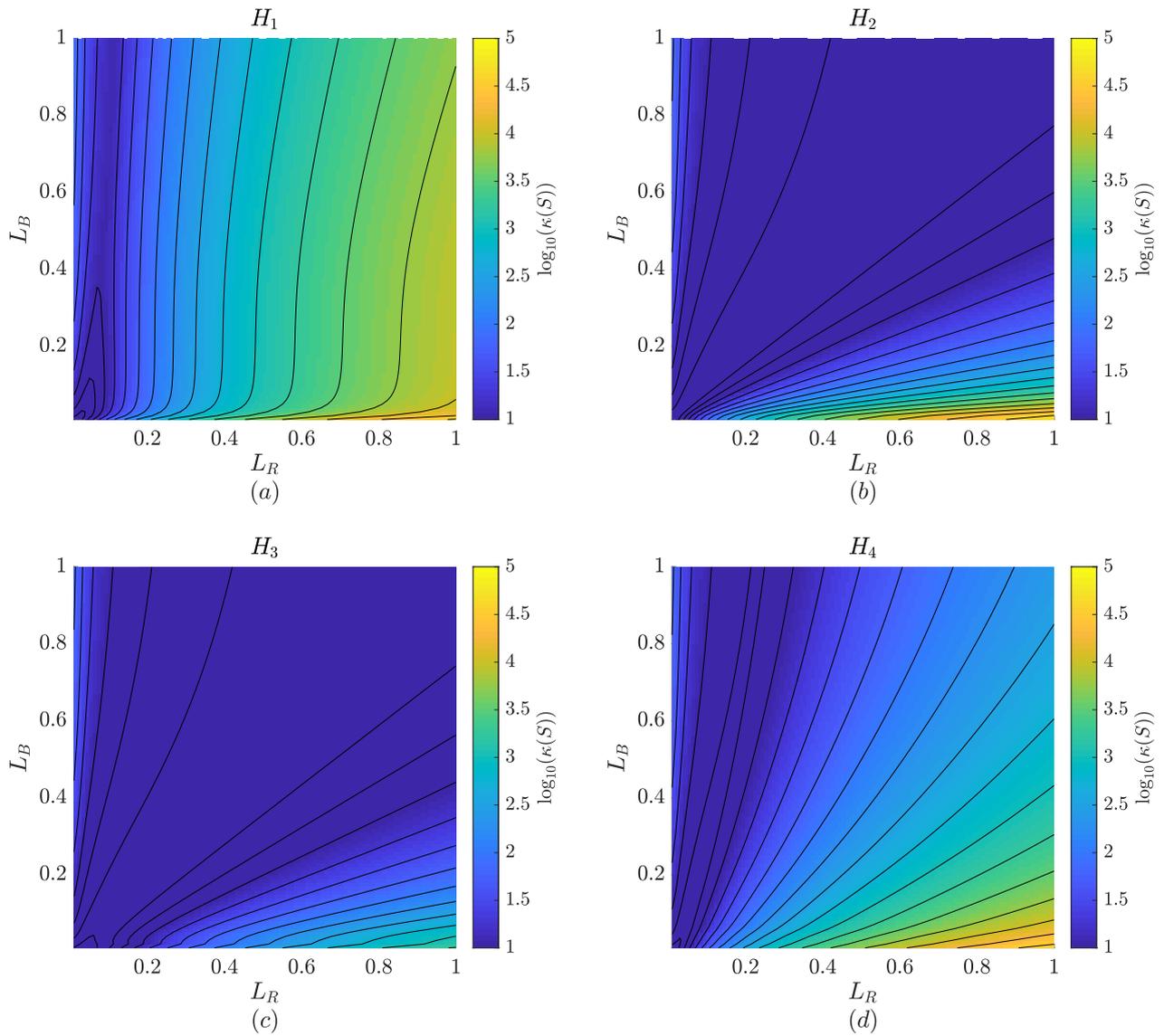}
		\caption[Changes to $\kappa(\Sh)$ with $L_R$, $L_B$ and $\bfH$]{Change to $\kappa(\Sh)$  with changes in $L_R$, $L_B$ for (a) $\bfH_1$, (b) $\bfH_2$, (c) $\bfH_3$ and (d) $\bfH_4$. The colour map is shown on a logarithmic scale which is standardised for all figures. Contours range from $\log_{10}(\kappa(\Sh)) =0.25$ to $\log_{10}(\kappa(\Sh))=5$ with a contour interval of $0.25$.\label{fig:6:heatmap}}
	\end{figure}

	Figure \ref{fig:6:heatmap} shows how the condition number of the preconditioned Hessian \eqref{eq:6:Sh} changes with the lengthscales of $\bfB$ and $\bfR$ for different choices of $\bfH$. 
	For $\bfH_1$,  increasing $L_R$ increases the value of $\kappa(\Sh)$. Changes with $L_B$ are much smaller, but increases to $L_B$ lead to a slight decrease in $\kappa(\Sh)$.
	For both $\bfH_2$ and $\bfH_3$, large values of $\kappa(\Sh)$ occur for very  large values of $L_R$ and small values of $L_B$. For a fixed value of $L_R$, increasing $L_B$ results in a rapid decrease in the value of $\kappa(\Sh)$. {For small fixed values of $L_R$ ($L_R<0.1$), this decrease is followed by a slow increase to $\kappa(\Sh)$  with increasing $L_B$.} The minimum value of $\kappa(\Sh)$ occurs when $L_R=L_B$; in this case $\bfH\bfB\bfH^T=\bfR$ to machine precision for both $\bfH_2$ and $\bfH_3$.  
	The qualitative behaviour for $\bfH_2$ and $\bfH_3$ is very similar, with smaller values of $\kappa(\Sh)$ for $\bfH_3$ than $\bfH_2$.  
	This is also the case in the unpreconditioned setting \cite{tabeart17a}, and occurs as $\bfH_3$ can be considered as a smoothed version of $\bfH_2$. 
	Qualitatively the behaviour for $\bfH_4$ is a compromise between $\bfH_1$ and $\bfH_2$; we can reduce $\kappa(\Sh)$ by increasing $L_B$ or decreasing $L_R$. 
	In the unpreconditioned case decreasing either lengthscale always reduces $\kappa(\Sh)$. However, in the preconditioned setting the ratio between background and observation lengthscales is important, meaning that for some cases increasing $L_B$ or $L_R$ will reduce $\kappa(\Sh)$.

	\begin{landscape}
		\begin{figure}
			\centering
			\includegraphics[trim=25mm 8mm 25mm 12mm, clip, width=1.3\textwidth]{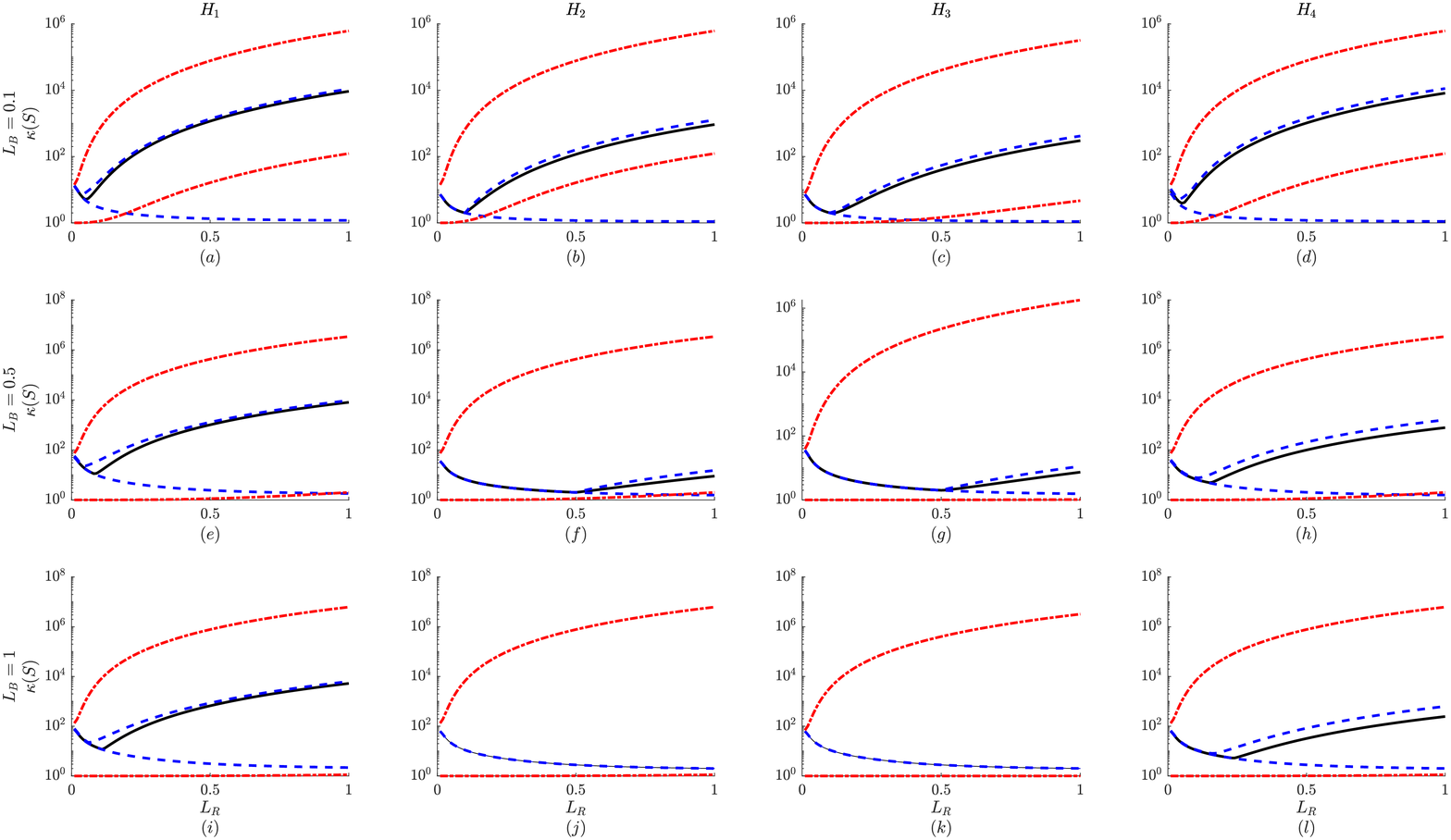}
			\caption[Changes to bounds on $\kappa(\Sh)$ with $L_R$, $L_B$ and $\bfH$]{Bounds and value of $\kappa(\Sh)$ for  (a,e,i) $\bfH_1$, (b,f,j) $\bfH_2$, (c,g,k) $\bfH_3$ and (d,h,l) $\bfH_4$ as a function of $L_R$. Blue dashed lines denote the bounds given by \eqref{eq:6:habenbound}, red dot-dashed lines denote the upper bound and first term in the lower bound of \eqref{eq:6:factoredprecond}. The solid black line denotes the value of $\kappa(\Sh)$ calculated using the $cond$ command in Matlab 2018b \cite{MATLAB:2018}. The different rows correspond to different values of $L_B$.  For (j) and (k) the upper and lower bounds of \eqref{eq:6:habenbound} are equal to $\kappa(\bfS)$ for all values of $L_R$ by the result of Corollary \ref{cor:6:BnormHRH}, and hence appear as a single line. \label{fig:6:Hbounds}}
		\end{figure}
	\end{landscape}
	
Figure \ref{fig:6:Hbounds} shows the value of $\kappa(\Sh)$, terms in the bounds \eqref{eq:6:factoredprecond},  and the bounds \eqref{eq:6:habenbound} for various combinations of $\bfH$, $\bfR$ and $\bfB$. The second term in the lower bound \eqref{eq:6:factoredprecond}, given by 
	$1+ {\lambda_{1}(\bfH\bfH^T)\lambda_N(\bfB)}({\lambda_1(\bfR)})^{-1}$,
	is not shown, as it performs worse than the first term of \eqref{eq:6:factoredprecond}, given by 
	$	1+ \lambda_{p}(\bfH\bfH^T)\lambda_N(\bfB)({\lambda_p(\bfR)})^{-1}$,
	for all parameter combinations studied.  
	Both the upper and lower bounds of \eqref{eq:6:factoredprecond} increase with $L_R$.  They represent the increase in $\kappa(\Sh)$  which occurs for $L_R\ge L_B$ for $\bfH_2$ and $\bfH_3$ and for larger values of $L_R$ for $\bfH_1$ and $\bfH_4$. The initial decrease of $\kappa(\Sh)$ with increasing $L_R$ is not represented by the bounds of \eqref{eq:6:factoredprecond}. Although some of the qualitative behaviour is well represented, the bounds are very wide.  Notably for larger values of $L_B$ the lower bound given by \eqref{eq:6:factoredprecond} is very close to $1$ for all values of $L_R$. In contrast, the bounds given by \eqref{eq:6:habenbound} represent the initial decrease in $\kappa(\Sh)$ for small values of $L_R$ well, both qualitatively and quantitatively. The upper bound of \eqref{eq:6:habenbound} then increases with increasing $L_R$ and remains tight for all parameter combinations. The lower bound of \eqref{eq:6:habenbound} is monotonically decreasing, and hence does not represent the behaviour  of $\kappa(\Sh)$ well for larger values of $L_B$ and $L_R$. We note that for $\bfH_2$ and $\bfH_3$ the upper and lower bounds of \eqref{eq:6:habenbound} are equal for $L_B>L_R$. This results from Corollary \ref{lem:6:circulantequivbound} as 
	$\bfH\bfB\bfH^T$ is circulant when $\bfH= \bfH_2$ or $ \bfH=\bfH_3$ and all entries in the product $\bfR^{-1/2}\bfH\bfB\bfH^T\bfR^{-1/2}$ are  positive for $L_B\ge L_R$. For panels (j) and (k) this means that the bounds given by \eqref{eq:6:habenbound} are equal to $\kappa(\Sh)$ for all plotted values of $L_R$
	
	Comparing the bounds given by \eqref{eq:6:factoredprecond} and \eqref{eq:6:habenbound}, we find that the upper bound of \eqref{eq:6:habenbound} performs better for all parameters studied. The best lower bound depends on the choice of $L_B$ and $L_R$: for lower values of $L_B$ and larger values of $L_R$  the first term of \eqref{eq:6:factoredprecond} is the tightest. Otherwise the bound given by \eqref{eq:6:habenbound} yields the tightest bound in this setting. 
	Although the bounds given by \eqref{eq:6:habenbound} represent the behaviour of $\kappa(\widehat{\bfS})$ well, we note that the numerical framework considered here has a very specific structure that is unlikely to occur in practice. Observation operators are likely to be much less smooth and have less regular structure: e.g. observations may not occur at the location of state variables, observation and state variables may not be evenly spaced, data may be missing, leading to different observation networks at different times or time windows. This may make a difference to the performance of both sets of bounds.  
	
	\begin{figure}
		\centering
		
		\includegraphics[trim = 20mm 0mm 18mm 0mm, clip, width=0.95\textwidth]{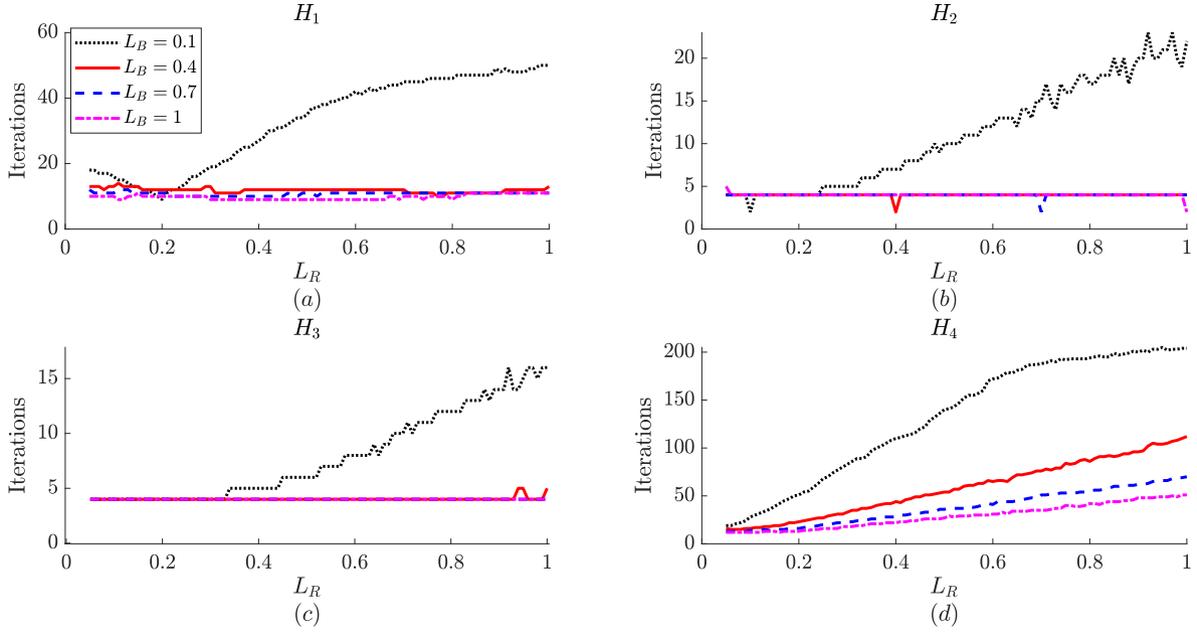} 
		
		\caption[Changes to convergence of conjugate gradient with $L_B$, $L_R$ and $\bfH$]{Number of iterations required for a conjugate gradient method to converge for changing values of $L_R$ and $L_B$ for (a) $\bfH_1$, (b) $\bfH_2$, (c) $\bfH_3$ and (d) $\bfH_4$. {\color{black} Note the difference in the y-axis values for each of the subplots.}} 
		\label{fig:6:CGconvergence} 
	\end{figure}

	We now consider how altering the data assimilation system affects the convergence of a conjugate gradient method for the problem introduced in Section \ref{sec:6:CGdesign}. 	Figure \ref{fig:6:CGconvergence} shows how convergence of the conjugate gradient problem changes with $L_B$, $L_R$ and $\bfH$. For all choices of $\bfH$ the largest number of iterations occurs when $L_R$ is large and $L_B$ is small. 	
	{\color{black}Similarly to the unpreconditioned case \cite{tabeart17a}}, we see that for many cases $\kappa(\Sh)$ is a good proxy for convergence: for $\bfH_2$, $\bfH_3$ and $\bfH_4$ reductions in $\kappa(\Sh)$ and the number of iterations required for convergence occur for the same changes to $L_R$ and $L_B$. 
The main difference in behaviour is seen for $\bfH_1$, where increasing $L_R$ increases $\kappa(\Sh)$ for all choices of $L_B$, but makes no difference to the number of iterations required for convergence for $L_B\ge0.4$.

	\begin{figure}
		\centering
		
		\includegraphics[width=0.99\textwidth, trim=36mm 0mm 43mm 0mm, clip]{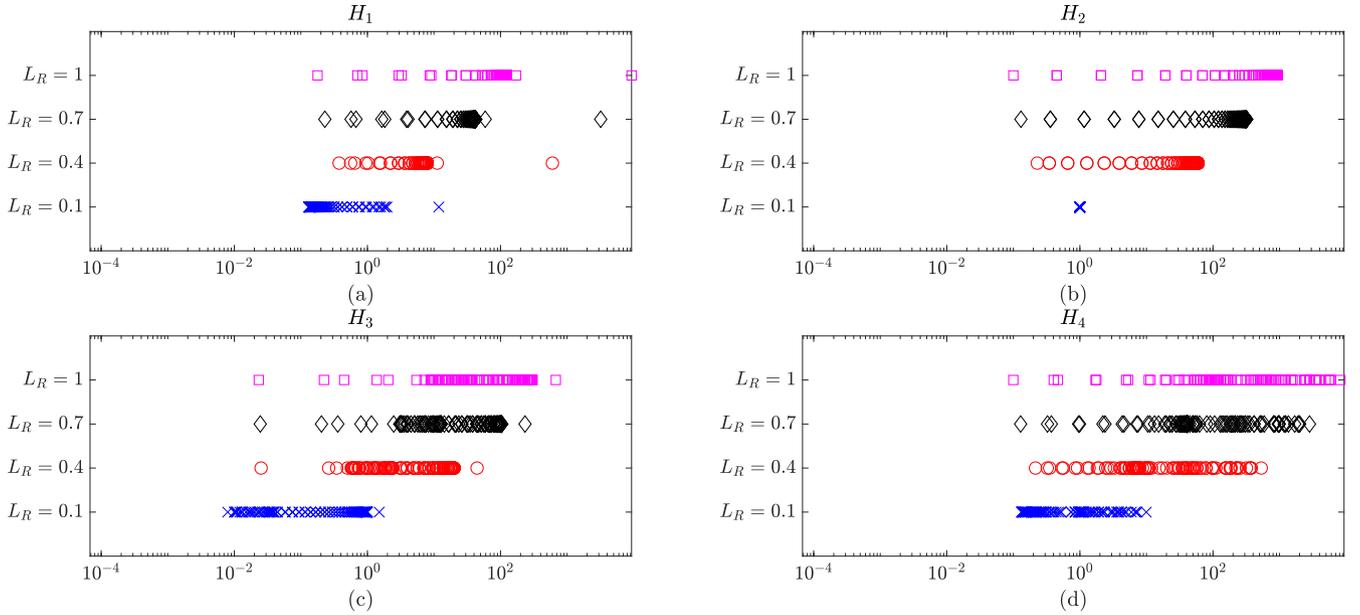} 	
		\caption{Non-zero eigenvalues of $\bfB^{1/2}\bfH^T\bfR^{-1}\bfH\bfB^{1/2}$ for  $L_B=0.1$ and $L_R=0.1$ (crosses), $L_R=0.4$ (circles), $L_R=0.7$ (diamonds) and $L_R=1$ (squares) for (a) $\bfH_1$, (b) $\bfH_2$, (c) $\bfH_3$ and (d) $\bfH_4$. Note the x-axis is plotted with a logarithmic scale. } 
		\label{fig:6:Sevals} 
	\end{figure}
	
	\begin{figure}
		\centering
		
		\includegraphics[width=0.99\textwidth, trim=36mm 0mm 43mm 0mm, clip]{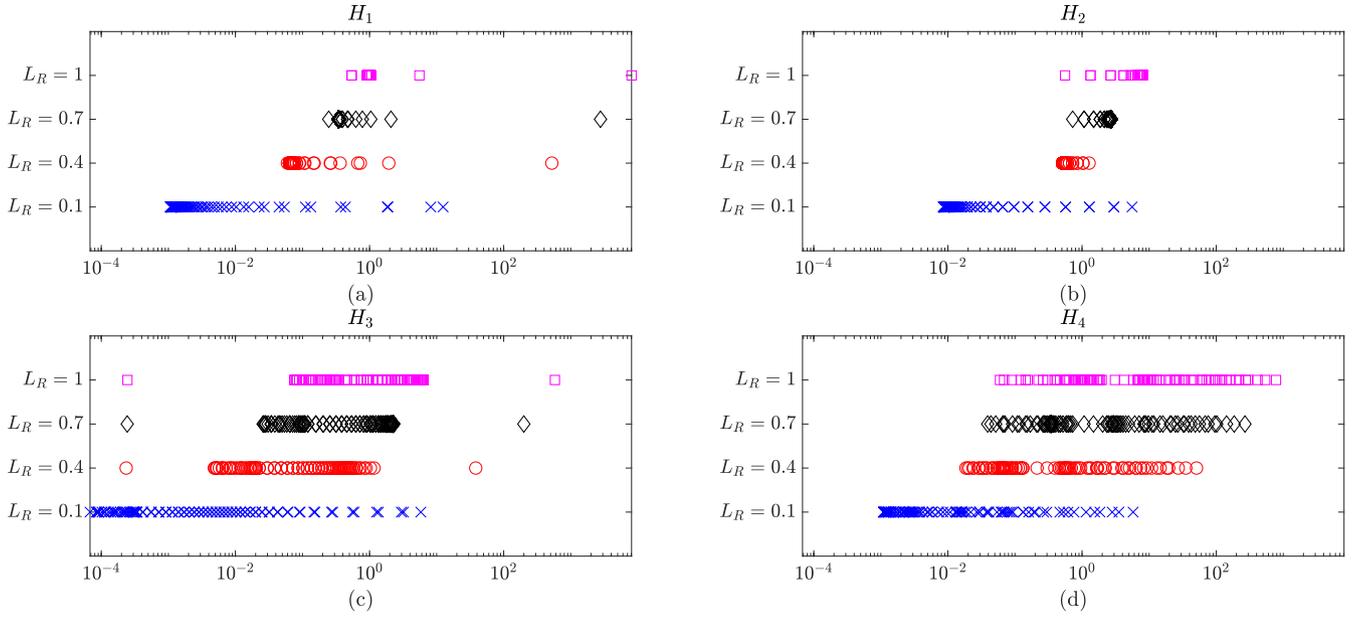} 	
		\caption{Non-zero eigenvalues of $\bfB^{1/2}\bfH^T\bfR^{-1}\bfH\bfB^{1/2}$ for  $L_B=0.5$ and $L_R=0.1$ (crosses), $L_R=0.4$ (circles), $L_R=0.7$ (diamonds) and $L_R=1$ (squares) for (a) $\bfH_1$, (b) $\bfH_2$, (c) $\bfH_3$ and (d) $\bfH_4$. Note the x-axis is plotted with a logarithmic scale. } 
		\label{fig:6:Sevals0.5} 
	\end{figure}

	This difference can be explained by considering the full distribution of the eigenvalues of $\Sh$ rather than just the condition number.	Convergence of the conjugate gradient method depends on the distribution of the entire spectrum, and we expect faster convergence to occur where eigenvalues are clustered (see Theorems 38.3, 38.5 of Trefethen and Bau \cite{TrefethenLloydN.LloydNicholas1997Nla}, and Theorem 38.4 of Gill, Murray and Wright \cite{Gill86}). 
	{\color{black} 	The eigenvalues of the full Hessian are given by $1+\lambda(\bfB^{1/2}\bfH^T\bfR^{-1}\bfH\bfB^{1/2})$, and $N-p$ further unit eigenvalues. Figure \ref{fig:6:Sevals} shows the non-zero eigenvalues of the low-rank update to the identity, $\bfB^{1/2}\bfH^T\bfR^{-1}\bfH\bfB^{1/2}$, for  $L_B=0.1$ and $L_R=0.1,0.4,0.7, 1$.  For all choices of $\bfH$ increasing $L_R$ leads to an increase in the maximum eigenvalue of the product. Additionally, the spectrum is distributed smoothly with few clusters, meaning that the condition number is a good indicator for convergence of a conjugate gradient method. This explains why  increasing $L_R$ for $L_B=0.1$  leads to an increase in the number of iterations required for convergence for all choices of $\bfH$.
		
		Figure \ref{fig:6:Sevals0.5}  shows the non-zero eigenvalues of the low-rank update to the identity, $\bfB^{1/2}\bfH^T\bfR^{-1}\bfH\bfB^{1/2}$, for  $L_B=0.5$ and $L_R=0.1,0.4,0.7, 1$. Although the maximum eigenvalue of the product gets larger with increasing $L_R$ for all choices of $\bfH$, we see increased clustering of the remaining eigenvalues about $1$ for $\bfH_1,\bfH_2,$ and $ \bfH_3$. Therefore, for these parameter choices the condition number of the Hessian does not represent convergence of a conjugate gradient method well.   
		For $\bfH_4$ no such clustering is observed, which explains why increasing $L_R$ for all values of $L_B$ leads to slower convergence of the conjugate gradient method for this observation operator. The clustering occurs due to the very regular structures of $\bfH_1,\bfH_2$ and $\bfH_3$.  Therefore if the observation operator is less regular, as may be expected in realistic observing networks \cite{Guillet19}, the condition number is more likely to be a good proxy for convergence of a conjugate gradient method in this setting.
	} 
	
	We conclude that in this framework changing $L_B$ has a larger effect on convergence of a conjugate gradient method than changing $L_R$. {\color{black}This contrasts with the unpreconditioned case, where changes to both $L_B$ and $L_R$ had a large impact on convergence \cite{tabeart17a}}.  For all choices of observation operator, small values of $L_B$ lead to poor convergence. Although changing $L_R$ impacts $\kappa(\Sh)$, due to an increase in clustered eigenvalues these changes do not always affect the convergence of a conjugate gradient method.
	Overall   the condition number is a good proxy for convergence in this framework.

	\section{Conclusions}\label{sec:6:Conclusions}

	The inclusion of correlated observation errors in data assimilation is important for high resolution forecasts \cite{fowler17,rainwater15}, and to ensure we make the best use of existing data \cite{Michel18,stewart13,simonin18}. However, multiple studies have found issues with convergence of data assimilation routines when introducing correlated observation error covariance (OEC) matrices \cite{weston11,campbell16,bormann16}. Earlier work considers the preconditioned data assimilation problem in the case of uncorrelated OEC matrices \cite{haben11c}. In this paper we study the effect of introducing correlated OEC matrices on the conditioning and convergence of the preconditioned variational data assimilation problem. This extends the theoretical and numerical results of a previous study by Tabeart et al. \cite{tabeart17a} that considered the use of correlated OEC matrices in the unpreconditioned variational data assimilation framework. 
	
	In this paper, we developed bounds on the condition number of the Hessian of the preconditioned variational data assimilation problem, {\color{black} for the case that there are fewer observations than state variables}. We then studied these bounds numerically in an idealised framework. We found that:
	
	\begin{itemize}
		\item As in the unpreconditioned case, decreasing the observation error variance or increasing the background error variance increases the condition number of the Hessian. 
		\item The minimum eigenvalue of the OEC matrix appears in both the upper and lower bounds. This was also true for the unpreconditioned case. 
	\item For a fixed lengthscale of the observation (background) error covariance matrix, $L$,  the condition number of the Hessian is smallest when the lengthscale of the background (observation) error covariance matrix is also equal to $L$.  
		This is in contrast to the unpreconditioned case, where  for a fixed lengthscale of the observation (background) error covariance matrix, the condition number of the Hessian is smallest when the lengthscale of the background (observation) error covariance is minimised.
		\item 	 Our new lower bound represented the qualitative behaviour better than an existing bound for some cases. 
		The upper bound from Haben \cite{haben11c} was shown to be tight for all parameter choices. We proved that under additional assumptions the upper and lower bounds from Haben \cite{haben11c} are equal. 
		\item For most cases the conditioning of the Hessian performed well as a proxy for the convergence of a conjugate gradient method. However in some cases, {\color{black} clustered} eigenvalues (induced by the specific structure of the numerical framework) meant that convergence was much faster than predicted by the conditioning. 
	\end{itemize}
	We remark that our findings about {\color{black} clustered} eigenvalues occur as our numerical framework has very specific structures. In particular, the eigenvectors of the background and observation error covariance matrices are strongly related. Other experiments not presented in this paper considered the use of the Laplacian correlation function for either or both of the observation and background error covariance matrices \cite{haben11c}. Qualitative conclusions were very similar to those shown in Section \ref{sec:6:Numerics}, even though the negative entries of the Laplacian correlation function do not satisfy the additional assumptions required for the bounds to be equal.
	In applications, we are likely to have more complicated observation operators, and the background and observation error covariance matrices are less likely to have complementary structures. Satellite observations for NWP often have interchannel correlation structures that are different from the typical spatial correlations of background error covariance matrices \cite{weston14,stewart13}. We also note that our state variables were evenly distributed and homogeneous, which will not be the case for  non-uniform grids. 
	
	In the unpreconditioned case using a similar numerical framework Tabeart et al. \cite{tabeart17a} found that improving the conditioning of the background or observation error covariance matrix separately would always decrease $\kappa(\Sh)$. The preconditioned system is more complicated; in some cases decreasing the condition number $\kappa(\bfB)$ or $\kappa(\bfRh)$ increases the condition number $\kappa(\Sh)$.
	We expect the relationship between each of the constituent matrices to be complicated for more general problems. This is relevant for practical applications, as estimated observation error covariance matrices 
	typically need to be treated via reconditioning methods before they can be used \cite{weston11,bormann16}. Currently the use of reconditioning methods is heuristic \cite{tabeart17c}, meaning that there may be flexibility to select a treated matrix that will result in faster convergence in some cases. However, popular reconditioning techniques work by increasing small eigenvalues of the observation error covariance matrix. In the preconditioned setting, such techniques will not automatically reduce the value of $\kappa(\Sh)$, due to the multiplication of background and observation error covariances. This means that reconditioning techniques may perform differently for the preconditioned data assimilation problem than in the unpreconditioned setting.

	{\color{black} Although the numerical experiments in this paper consider a limited choice of matrices and parameters, we note that the theory and bounds presented in this work are general and apply to any choice of covariance matrices $\bfB$ and $\bfR$, and any linear observation operator (or generalised observation operator in the case of 4D-Var). We could consider the numerical results presented here as a `best case' due to the circulant structure of both covariance matrices. For more general choices of $\bfB$ and $\bfR$ any eigenvalue clustering is likely to be less extreme, and hence conditioning may be more influential for the convergence of a conjugate gradient method. Increased eigenvalue clustering occurred for observation operators with regular structure, whereas in practice the `randomly observed' experiment is more realistic. For the 4D-Var problem the generalised observation operator $\widehat{\bfH}$ also accounts for model evolution, and hence the structure of the linearised model is also expected to be important when considering clustering and convergence of a conjugate gradient problem. Previous work has also shown that for the unpreconditioned problem, the qualitative behaviour of an operational system \cite{tabeart17b} largely followed the linear theory \cite{tabeart17a}. Similarly, for the case of uncorrelated observation error covariance matrices, the behaviour of  preconditioned 4D-Var experiments broadly coincided with theory from the linear setting \cite{haben11b,haben11c}. This indicates that conclusions arising from the study of linear data assimilation problems can often provide insight for a wider range of practical implementations, even if theoretical results are not directly applicable. }
	\section*{Acknowledgements}
	
	This work is funded in part by the EPSRC Centre for Doctoral Training in Mathematics of
	Planet Earth, the NERC Flooding from Intense Rainfall programme (NE/K008900/1), the EPSRC
	DARE project (EP/P002331/1), the NERC National Centre for Earth Observation and EPSRC grant (EP/S027785/1).

	\bibliographystyle{acm}
	\bibliography{wileyNJD-VANCOUVER}

\end{document}